\documentclass{amsart}
\usepackage{hyperref}
\usepackage{amssymb}
\usepackage{graphicx}
\usepackage{cite}

\usepackage[centertags]{amsmath}
\usepackage{amsfonts, amsbsy}
\usepackage{latexsym, amsthm}

\begin{document}
\title[\hfil Impulsive hematopoiesis models with delays]
{On the existence and exponential attractivity of a unique positive almost periodic solution 
to an impulsive hematopoiesis model with delays}

\author[T.T. Anh, T.V. Nhung \& L.V. Hien]
{Trinh Tuan Anh, Tran Van Nhung  and  Le Van Hien}  

\address{Trinh Tuan Anh \newline
Department of Mathematics, Hanoi National University of Education\\
136 Xuan Thuy Road, Hanoi, Vietnam}
\email{anhtt@hnue.edu.vn}

\address{Tran Van Nhung \newline
Vietnam State Council for Professor Promotion, No. 1, Dai Co Viet Road, Hanoi, Vietnam}
\email{tvnhung@moet.edu.vn}

\address{Le Van Hien\newline
Department of Mathematics, Hanoi National University of Education\\
136 Xuan Thuy Road, Hanoi, Vietnam}
\email{hienlv@hnue.edu.vn}

\thanks{Accepted for publication in AMV}
\maketitle

\begin{abstract}
In this paper, a generalized model of hematopoiesis with delays and impulses is considered.
By employing the contraction mapping principle and a novel type of impulsive delay inequality, 
we prove the existence of a unique positive almost periodic solution of the model.
It is also proved that, under the proposed conditions in this paper, the unique positive almost periodic
solution is globally exponentially attractive. A numerical example is given to illustrate
the effectiveness of the obtained results.
\end{abstract}
\keywords{Hematopoiesis model; almost periodic solution; impulsive systems.}

\numberwithin{equation}{section}
\newtheorem{theorem}{Theorem}[section]
\newtheorem{corollary}[theorem]{Corollary}
\newtheorem{lemma}{Lemma}[section]

\theoremstyle{definition}
\newtheorem{definition}{Definition}[section]
\newtheorem{remark}{Remark}[section]
\newtheorem{example}{Example}[section]
\allowdisplaybreaks

\section{Introduction}

The nonlinear delay differential equation
\begin{equation}\label{eq1}
\dot x(t)=-ax(t)+\frac{b}{1+x^n(t-\tau)},\quad n>0,
\end{equation}
where $a,b,\tau$ are positive constants, proposed by Mackey and Glass \cite{MG}, 
has been used as an appropriate model for the dynamics of hematopoiesis 
(blood cells production) \cite{BBI2,GL,MG}. 
In medical terms, $x(t)$ denotes the density of mature cells in blood circulation 
at time $t$ and $\tau$ is the time delay between the production of immature 
cells in the bone marrow and their maturation for release in circulating bloodstream. 

As we may know, the periodic or almost periodic phenomena are popular in various 
natural problems of real world applications \cite{A,BBI2,F,HLSC,LZ,LD,Sm,WLX,WZh}. 
In comparing with periodicity, almost periodicity is more 
frequent in nature and much more complicated in studying for such model \cite{Stamov,WZ}. 
On the other hand, many dynamical systems describe the real phenomena depend on the history 
as well as undergo abrupt changes in their states. This kind of models are best described by impulsive 
delay differential equations \cite{BS,SP,Stamov}. A great deal of effort from researchers has been devoted to study the 
existence and asymptotic behavior of almost periodic solutions of \eqref{eq1} and its generalizations due to 
their extensively realistic significance. We refer the reader to \cite{HW,Liu,GYZ,WZ,WLZ,ZYW} and the references therein. 
Particularly, in \cite{WZ}, Wang and Zhang investigated the existence, 
nonexistence and uniqueness of positive almost periodic solution of the following model
\begin{equation}\label{eq2}
\dot x(t)=-a(t)x(t)+\frac{b(t)x(t-\tau(t))}{1+x^n(t-\tau(t))},\quad n>1,
\end{equation}
by using a new fixed point theorem in cone. Very recently, using a fixed point theorem 
for contraction mapping combining with the Lyapunov functional method, 
Zhang et al. \cite{ZYW} obtained sufficient conditions for the 
existence and exponential stability of a positive almost periodic 
solution to a generalized model of \eqref{eq1}
\begin{equation}\label{eq3}
\dot x(t)=-a(t)x(t)+\sum_{i=1}^m\frac{b_i(t)}{1+x^n(t-\tau_i(t))},\quad n>0.
\end{equation}
By employing a novel argument, a delay-independent criteria was established
in \cite{Liu} ensuring the existence, uniqueness, and global exponential stability of positive almost periodic solutions of
a non-autonomous delayed model of hematopoiesis with almost periodic coefficients and delays.
In \cite{ANS}, Alzabut et al. considered the following model of hematopoiesis with impulses
\begin{equation}\label{eq4}
\begin{aligned}
&\dot x(t)=-a(t)x(t)+\frac{b(t)}{1+x^n(t-\tau)},\; t\geq 0,\; t\neq t_k,\\
&\Delta x(t_k)=\gamma_kx\left(t_k^-\right)+\delta_k,\; k\in \mathbb{N},
\end{aligned}
\end{equation}
where $t_k$ represents the instant at which the density suffers an increment 
of $\delta _k$ unit and $\Delta x(t_k)=x\left(t_k^+\right)-x\left(t^-_k\right)$. 
The density of mature cells in blood circulation decreases at prescribed 
instant $t _k$ by some medication and it is proportional to the density 
at that time $t ^-_k$. By employing the contraction mapping principle 
and applying Gronwall-Bellman's inequality, sufficient conditions which 
guarantee the existence and exponential stability of a positive 
almost periodic solution of system \eqref{eq4} were given in \cite{ANS} as follows.

\begin{theorem}[\cite{ANS}]\label{thm1}
 Assume that
\begin{itemize}
\item[\rm(C1)] The function $a\in C\left(\mathbb{R}^+,\mathbb{R}^+\right)$  
is almost periodic in the sense of Bohr and there exists a 
positive constant $\mu$ such that $a(t)\geq \mu$.
\item[\rm(C2)] The sequence $\{\gamma_k\}$ is almost periodic and 
$-1\leq \gamma_k \leq 0,\ k\in \mathbb{N}$.
\item[\rm(C3)] The sequences $\{t_k^p\}$ are uniformly almost 
periodic and there exists a positive constant $\eta $ such that $\inf_{k\in \mathbb{N}}t^1_k=\eta$, 
where $0<\sigma\leq t_k<t_{k+1}, \forall k\in\mathbb{N}$, $\lim_{k\to \infty}t_k=\infty$ 
and $t_k^p=t_{k+p}-t_k, k,p\in \mathbb{N}$.
\item[\rm(C4)] The function $b\in C\left(\mathbb{R}^+,\mathbb{R}^+\right)$ is almost 
periodic in the sense of Bohr, $b(0)=0$, and there exists a positive 
constant $\nu$ such that $\sup_{t\in \mathbb{R}^+}|b(t)|< \nu$.
\item[\rm(C5)] The sequence $\{\delta_k\}$ is almost periodic and 
there exists a constant $\delta>0$ such that $\sup_{k\in\mathbb{N}}|\delta_k|<\delta$.
\end{itemize}

 If  $\nu <\mu$, then equation \eqref{eq4} has a unique positive almost periodic solution.
\end{theorem}

Unfortunately, the above theorem is incorrect.
For this, let us consider the following example.

\begin{example}\label{exam1}
Consider the following equation
\begin{equation}\label{eq5}
 \dot x(t)= - x(t), \quad t\geq 0, t\neq k \in \mathbb{N},\quad \Delta x(k)=-1, \; k\in \mathbb{N}.
\end{equation}
 
Note that \eqref{eq5} is a special case of \eqref{eq4}.
Moreover, we can easily see that equation \eqref{eq5} satisfies conditions (C1)-(C5),
where $t_k=k, \gamma_k=0$ and $\delta_k=-1$.

Suppose that system \eqref{eq5} has a positive almost periodic solution $x^*(t)$. It is obvious that 
\begin{equation*}
x^*(t)=e^{-t}x^*(0)-\sum_{k\in\mathbb{N},\; k\leq t}e^{-(t-k)},\quad t>0.
\end{equation*}

For any positive integer $n$, we have
\begin{equation*}
0<x^*(n)=e^{-n}x^*(0)-e^{-1}\left(\frac{1-e^{-n}}{1-e^{-1}}\right)
\longrightarrow \frac{-e^{-1}}{1-e^{-1}}<0 \text{ as } n\to \infty
\end{equation*}
which yields a contradiction. This shows that \eqref{eq5} has no positive almost 
periodic solution. 
Thus, Theorem \ref{thm1} is incorrect, and Theorem 3.2 in \cite{ANS} is also incorrect.
\end{example}

Motivated by the aforementioned discussions, in this paper we consider a generalized 
model of hematopoiesis with delays, harvesting terms \cite{BBI1,LM,Long} and impulses of the form
\begin{equation}\label{eq6}
\begin{split}
&\begin{aligned}\dot x(t)=-a(t)x(t)&+\sum_{i=1}^m\bigg[\frac{b_i(t)}{1+x^{\alpha_i}(t-\tau_i(t))}\\
&+c_i(t)\int_0^{T}\frac{v_i(s)}{1+x^{\beta_i}(t-s)}ds-H_i(t,x(t-\sigma_i(t)))\bigg],\quad t\neq t_k,
\end{aligned}\\
&\Delta x(t_k)=x\left(t_k^+\right)-x\left(t^-_k\right)
=\gamma_kx\left(t_k^-\right)+\delta_k, \quad k\in \mathbb{Z},
\end{split}
\end{equation}
where $m$ is given positive integer, 
$a, b_i,c_i:\mathbb{R}\to \mathbb{R}, i\in\underline{m}:=\{1,2,\ldots,m\}$, 
are nonnegative functions; $H_i:\mathbb{R}\times \mathbb{R}\to \mathbb{R}_+, i\in\underline{m},$ 
are nonnegative functions represent harvesting terms; 
$\tau_i(t),\sigma _i(t)\geq 0, i\in\underline{m}$, are time delays; 
$\alpha _i, \beta_i, i\in\underline{m},$ are positive numbers and $T>0$ is a constant;  
$\gamma _k, \delta_k, k\in \mathbb{Z},$ are constants; 
$v_i(t), i\in\underline{m},$ are nonnegative integrable functions on $[0,T]$ with 
$\int_0^Tv_i(s)ds=1$; $\{t_k\}, k\in\mathbb{Z}$, is an increasing sequence 
involving the fixed impulsive points with 
$\lim_{k\to \pm\infty}t_k=\pm\infty$. 

The main goal of the present paper is to establish conditions for the 
existence of a unique positive almost periodic 
solution of  model \eqref{eq6}. It is also proved that, under the proposed
conditions, the unique positive almost periodic solution of \eqref{eq6} is
globally exponentially attractive. 

The rest of this paper is organized as follows. Section 2 introduces some notations, basic definitions
and technical lemmas. Main results on the existence and 
exponential attractivity of a unique positive almost periodic solution of \eqref{eq6}
are presented in section 3. An illustrative example is given in section 4. 
The paper ends with the conclusion and cited references.

\section{Preliminaries}

Let $\{t_k\}_{k\in \mathbb{Z}}$ be a fixed sequence of real numbers 
satisfying $t_k<t_{k+1}$ for all $k\in \mathbb{Z}$, 
$\lim_{k\to \pm\infty}t_k=\pm\infty$. Let $X$ be an interval of 
$\mathbb{R}$, denoted by $PLC(X,\mathbb{R})$ the space of 
all piecewise left continuous functions $\phi: X\to \mathbb{R}$ 
with points of discontinuity of the first kind at $t=t _k$, $k\in \mathbb{Z}$. 

The following notations will be used in this paper. 
For bounded functions $f:\mathbb{R}\to \mathbb{R}$, 
$F:\mathbb{R}\times \mathbb{R}_+\to \mathbb{R}$ and a bounded sequence $\{ z_k\}$, we set
\begin{align*} 
 &f_L=\inf_{t\in \mathbb{R}}f(t),\; f_M=\sup_{t\in \mathbb{R}}f(t),\\
&F_L=\inf_{(t,x)\in \mathbb{R}\times\mathbb{R}_+}F(t,x),\; 
F_M=\sup_{(t,x)\in \mathbb{R}\times \mathbb{R}_+}F(t,x),\\ 
&z_L=\inf_{k\in \mathbb{Z}}z_k,\; z_M=\sup_{k\in \mathbb{Z}}z_k.
\end{align*}

The following definitions are borrowed from \cite{SP}.

\begin{definition}[\cite{SP,Stamov}]\label{def1}
The set of sequences $\{t_k^p\}$, where $t_k^p=t_{k+p}-t_k, p, k\in \mathbb{Z}$,  
is said to be uniformly almost periodic if for any positive number 
$\epsilon$, there exists a relatively dense set of $\epsilon$-almost 
periods common for all sequences.
\end{definition}

\begin{definition}[\cite{LZ,SP}]\label{def2}
A function $\phi \in PLC(\mathbb{R},\mathbb{R})$ is said to be 
almost periodic if the following conditions hold
\begin{itemize}
\item[(i)]  The set of sequences $\{t_k^p\}$ is uniformly almost periodic.
\item[(ii)] For any $\epsilon>0$, there exists $\delta=\delta(\epsilon)>0$ such that, 
if $t,\bar t$ belong to the same interval of continuity of $\phi (t)$, 
$|t-\bar t|<\delta$, then $|\phi (t)-\phi (\bar t)|<\epsilon$.
\item[(iii)] For any $\epsilon>0$, there exists a relatively 
dense set $\Omega$ of $\epsilon$-almost periods such that, if $\omega \in \Omega$ 
then $|\phi (t+\omega )-\phi (t)|<\epsilon$ for all $t\in \mathbb{R}$, 
$k\in \mathbb{Z}$ satisfying $|t-t_k|>\epsilon$. 
\end{itemize}
\end{definition}

For equation \eqref{eq6}, we introduce the following assumptions
\begin{itemize}
\item[(A1)]  The function $a(t)$ is almost periodic in the sense of Bohr and $a_L>0$.
\item[(A2)] The functions $ b_i(t),c_i(t), i\in\underline{m}$, 
are nonnegative and almost periodic in the sense of Bohr.
\item[(A3)] The function $H_i(t,x), i\in\underline{m}$, 
are bounded nonnegative and almost periodic in the sense of Bohr 
in $t\in \mathbb{R}$ uniformly in $x\in \mathbb{R}_+$ and 
there exist positive constants $L_i$ such that
\[
|H_i(t,x)-H_i(t,y)|\leq L_i|x-y|, \; \forall (t,x),(t,y)\in \mathbb{R}\times \mathbb{R}_+.
\]
\item[(A4)] The functions $\tau_i (t), \sigma_i(t), i\in\underline{m}$, 
are almost periodic in the sense of Bohr, $\dot\tau_i(t), \dot\sigma_i(t)$ are bounded,  
$\inf_{t\in \mathbb{R}}\left(1-\dot \tau_i(t)\right)>0$, 
$\inf_{t\in \mathbb{R}}\left(1-\dot \sigma_i(t)\right)>0$.
\item[(A5)] The sequence $\{\delta_k\}$ is almost periodic.
\item[(A6)] The sequence $\{\gamma_k\}$ is almost periodic satisfying
\[
\gamma_L >-1, \quad \Gamma_M=\sup_{p,q\in\mathbb{Z}, p\geq q}\Gamma(q,p)<\infty,\quad 
\Gamma_L=\inf_{p,q\in \mathbb{Z}, p\geq q}\Gamma(q,p)>0,
\]
where $\Gamma (q,p)=\prod\limits_{i=q}^p(1+\gamma_i), p\geq q$.
\item[(A7)] The set of sequences $\{t_k^p\}$ is uniformly 
almost periodic, $\eta =\inf_{k\in \mathbb{Z}}t_k^1>0$.
\end{itemize}

\begin{remark}\label{rm1}
It should be noted that model \eqref{eq6} includes \eqref{eq4} as a special case. For that model, 
assumptions  (A3), (A4) obviously be removed. 
Furthermore, we make assumption (A6) in order to correct condition (C2) in \cite{ANS}.
\end{remark}

The following lemmas will be used in the proof of our main results.

\begin{lemma}[\cite{SP}]\label{lem1}
Let Assumption {\rm(A7)} holds. Assume that functions  
$g_i(t), i\in\underline{m}$, are almost periodic in the sense of Bohr, 
a function $\phi (t)$ and sequences  $\{\delta_k\}$, $\{\gamma_k\}$ 
are almost periodic. Then for any $\epsilon>0$, there exist 
$\epsilon_1\in(0,\epsilon)$, relatively 
dense sets $\Omega\subset\mathbb{R}$, $\mathcal P\subset\mathbb{Z}$ such that
\begin{itemize}
\item[\rm(a1)] $|\phi (t+\omega)-\phi (t)|<\epsilon, t\in \mathbb{R}, |t-t_k|>\epsilon, k\in\mathbb{Z}$;
\item[\rm(a2)] $|g_i(t+\omega)-g_i(t)|<\epsilon,  t\in \mathbb{R},  i\in\underline{m}$;
\item[\rm(a3)] $|\gamma _{k+p}-\gamma _k|< \epsilon, 
|\delta _{k+p}-\delta _k|< \epsilon, |t^p_k-\omega|<\epsilon _1, 
\omega \in \Omega, p\in {\mathcal P}, k\in \mathbb{Z}.$
\end{itemize}
\end{lemma}

\begin{lemma}\label{lem2} 
Let Assumption {\rm(A7)} holds. Assume that functions $f_i(t,x), i\in\underline{m}$, 
are almost periodic in $t\in \mathbb{R}$ uniformly in $x \in \mathbb{R}$ 
in the sense of Bohr, a function $\phi (t)$ and sequences  
$\{\delta_k\}$, $\{\gamma_k\}$ are almost periodic. 
Then for any compact set ${\mathcal M}\subset \mathbb{R}$ and positive number 
$\epsilon$, there exist $\epsilon_1\in(0,\epsilon)$, relatively dense sets 
$\Omega\subset\mathbb{R}$, $\mathcal P\subset\mathbb{Z}$ such that 
\begin{itemize}
\item[\rm(b1)] $|\phi (t+\omega)-\phi (t)|<\epsilon, t\in \mathbb{R}, |t-t_k|>\epsilon, \omega\in\Omega$;
\item[\rm(b2)] $|f_i(t+\omega,x)-f_i(t,x)|<\epsilon, t\in \mathbb{R},
x\in {\mathcal M}, \omega\in\Omega, i\in\underline{m}$;
\item[\rm(b3)] $|\gamma _{k+p}-\gamma _k|< \epsilon, 
|\delta _{k+p}-\delta _k|< \epsilon, |t^p_k-\omega|<\epsilon _1, 
\omega \in \Omega, p\in {\mathcal P}, k\in \mathbb{Z}.$
\end{itemize}
\end{lemma}

\begin{proof}  The proof of this lemma is similar to the
proof of Lemma 2.1 in \cite{Stamov} so let us omit it here.
\end{proof}

\begin{lemma}\label{lem3} For given $\epsilon>\epsilon_1>0$, 
real number $\omega$ and integers $k,p$ such that 
$|t^p_k-\omega|<\epsilon_1$, if $|t-t_i|>\epsilon$ for all 
$i\in \mathbb{Z}$ and $t_{k-1}<t<t_k$ then 
$t_{k+p-1}<t+\omega <t_{k+p}$.
\end{lemma}

\begin{proof} The proof is straight forward, so let us omit it here.
\end{proof}

\begin{lemma}\label{lem4} Let Assumptions {\rm(A6)} and {\rm (A7)}  hold. 
If $p\in \mathbb{Z}$ satisfies $|\gamma_{i+p}-\gamma_i|\leq \epsilon$ for all $i\in \mathbb{Z}$, then 
\begin{equation*}
\left|\Gamma (n+p, k+p)-\Gamma (n,k)\right|
\leq\frac{\Gamma_M}{1+\gamma_L}(k-n+1)\epsilon,\quad \forall k, n\in\mathbb{Z}, k\geq n.
\end{equation*}
\end{lemma}

\begin{proof} Using the facts that 
$\left|e^u-e^v\right|\leq |u-v|\max\{e^u,e^v\}, \forall u,v\in\mathbb{R}$
and $\left|\ln(1+u)-\ln(1+v)\right|\leq \dfrac{1}{1+\min\{u,v\}}|u-v|, \forall u,v>-1$,
from (A6), we have 
\begin{align*}
\left|\Gamma (n+p,k+p)-\Gamma (n,k)\right|
&\leq\left|\exp\bigg(\sum_{i=n+p}^{k+p}\ln(1+\gamma_i)\bigg)
-\exp\bigg(\sum_{i=n}^{k}\ln(1+\gamma_i)\bigg)\right| \\
& \leq \frac{{\Gamma_M}}{1+\gamma_L}
\sum\limits_{i=n}^k|\gamma_{i+p} -\gamma_i| 
\leq \frac{{\Gamma_M}}{1+\gamma_L}(k-n+1)\epsilon,\ k\geq n .
\end{align*}
The proof is completed.
\end{proof}

\begin{lemma}\label{lem5} Let Assumption {\rm(A7)} holds. 
For any $\alpha>0$, $0<\epsilon < \eta/2$, we have 
\begin{equation*}
\sum_{t_k<t}e^{-\alpha (t-t_k)}\leq \frac{1}{1-e^{-\alpha \eta}},\quad 
\sum_{t_k<t}\int_{t_k-\epsilon}^{t_k+\epsilon}e^{-\alpha (t-s)}ds 
\leq 2 e^{\frac{1}{2}\alpha \eta}\frac{\epsilon}{1-e^{-\alpha\eta}}.
\end{equation*}
\end{lemma}

\begin{proof} The proof follows by some direct estimates and, thus, is omitted here. 
\end{proof}

Now, let $\sigma =\max_{i\in\underline{m}}\{T, \tau_{iM}, \sigma_{iM}\}$, 
then $0<\sigma <+\infty$. From biomedical significance, 
we only consider the initial condition
\begin{equation}\label{eq9}
x(s)=\xi(s)\geq 0, \ s\in [\alpha-\sigma, \alpha),\ \xi(\alpha)>0,\ 
\xi \in PLC([\alpha-\sigma,\alpha],\mathbb{R}).
\end{equation}

It should be noted that problem \eqref{eq6} and \eqref{eq9} 
has a unique solution $x(t)=x(t;\alpha,\xi)$ defined on 
$[\alpha-\sigma,\infty)$ which is piecewise continuous with points 
of discontinuity of the first kind, namely $t_k, k\in\mathbb{Z}$, 
at which it is left continuous and the following relations are satisfied \cite{SP}
\begin{equation*}
x\left(t^-_k\right)=x(t_k),\ \Delta x(t_k):=x\left(t_k^+\right)-x\left(t^-_k\right)
=\gamma_kx\left(t_k^-\right)+\delta_k.
\end{equation*}

Related to \eqref{eq6}, we consider the following linear equation
\begin{equation}\label{eq10}
 \dot y(t)=- a(t)y(t), \ t\neq t_k,\quad \Delta y(t_k)=\gamma_ky(t^-_k),\ k\in \mathbb{Z}.
\end{equation}

\begin{lemma}\label{lem6} Let Assumptions {\rm(A1), (A6)} and {\rm(A7)} hold. Then 
\begin{equation*}
Be^{-a_M(t-s)} \leq  H(t,s)\leq Ae^{-a_L(t-s)}, \ s \leq t,
\end{equation*}
where
\begin{equation}\label{eq11}
H(t,s)=\begin{cases} \exp\left(-\int_s ^t a(r)dr\right),\  
\text{ if } \ t_{k-1}<s \leq t \leq t_k,\\
\Gamma(n,k)\exp\left(-\int_s^t a(r) dr\right),\ \text{ if }  
\ t_{n-1}<s \leq t_n \leq t_k<t \leq t_{k+1}
\end{cases}
\end{equation}
is the Cauchy matrix of \eqref{eq10}, $A= \max\left\{\Gamma_M, 1\right\}$ 
and $B= \min\left\{\Gamma_L, 1\right\}$.
\end{lemma}

\begin{proof} The proof is straight forward from \eqref{eq11}, so let us omit it here.
\end{proof}

Similar to Lemma 36 in \cite{SP} and Lemma 2.6 in \cite{Stamov}
we have the following lemma.

\begin{lemma}\label{lem7} Let Assumptions {\rm(A1), (A6)} and {\rm(A7)} hold. 
Then, for given $0<\epsilon_1<\epsilon$, relatively dense sets $\Omega\subset \mathbb{R}$, 
$\mathcal{P}\subset\mathbb{Z}$, satisfying
\begin{itemize}
\item[\rm(c1)] $|a(t+\omega)-a(t)|<\epsilon, \ t\in \mathbb{R}, \omega \in \Omega$;
\item[\rm (c2)] $|\gamma_{k+p}-\gamma_k|<\epsilon,\ 
|t^p_k-\omega|<\epsilon_1, \omega \in \Omega,\ p\in{\mathcal P},\ k\in \mathbb{Z}$,
\end{itemize}
the following estimate holds
\begin{equation*}
\left|H(t+\omega, s+\omega)-H(t,s)\right|\leq \epsilon Me^{-\frac{1}{2}a_L(t-s)}, 
\end{equation*}
for any $\omega \in \Omega$, $t,s\in\mathbb{R}$ satisfying $t\geq s$, 
$|t-t_k|>\epsilon, \ |s-t_k| >\epsilon, k\in \mathbb{Z}$, where 
\begin{equation}\label{eq12}
M=\max\left\{\frac{2}{a_L}, \Gamma_M\left[\frac{2}{a_L}
+\frac{1}{1+\gamma_L}\left(1+\frac{2}{a_L\eta}\right)\right]\right\}.
\end{equation}
\end{lemma}

\begin{proof}

We divide the proof into two possible cases as follows.

\noindent
{\it Case 1:} $t_{k-1}<s\leq t\leq t_k$. By Lemma \ref{lem3}, 
$t_{k+p-1}<s+\omega\leq t+\omega <t_{k+p}$. 
Since $|a(t+\omega)-a(t)|\leq \epsilon$, $\forall t\in \mathbb{R}$, 
$\epsilon <\eta/2$ and $\frac12a_L(t-s)e^{-\frac{1}{2}a_L(t-s)}<1$, 
it follows from \eqref{eq11}, \eqref{eq12} that
\begin{equation}\label{eq13}
\begin{aligned}
|H(t+\omega,s+\omega)-H(t,s)|
&= \left|\exp\left(-\int_{s}^{t}a(r+\omega)dr\right)-\exp\left(-\int_{s}^{t}a(r)dr\right)\right| \\
& \leq e^{-a_L(t-s)}\int_s^t|a(r+\omega)-a(r)|dr\\
&\leq \frac{2}{a_L} \epsilon e^{-\frac{1}{2}a_L(t-s)}\leq \epsilon Me^{-\frac{1}{2}a_L(t-s)}.
  \end{aligned}
\end{equation}

\noindent
{\it Case 2:} $t_{n-1}<s \leq t_n \leq t_k<t \leq t_{k+1}$. Similarly, we have 
$$t_{n+p-1}<s+\omega<t_{n+p}\leq t+\omega <t_{k+p+1}.$$
By Lemma \ref{lem4}, from \eqref{eq11}-\eqref{eq13} we obtain
\begin{align*}
|H(t+\omega,s+\omega)-H(t,s)|&= \Gamma (n+p,k+p)\left|
\exp\left(-\int_{s+\omega}^{t+\omega}a(r)dr\right)
-\exp\left(-\int_{s}^{t}a(r)dr\right)\right|\\
&\quad + |\Gamma (n+p,k+p)-\Gamma (n,k)|\exp\left(-\int_{s}^{t}a(r)dr\right)\\
&\leq \frac{2\Gamma_M\epsilon}{a_L}e^{-\frac{1}{2}a_L(t-s)}
+\frac{\Gamma_M\epsilon}{1+\gamma_L}(k-n+1)e^{-a_L(t-s)}\\
&\leq \frac{2\Gamma_M\epsilon}{a_L}e^{-\frac{1}{2}a_L(t-s)}
+\frac{\Gamma_M\epsilon}{1+\gamma_L}\left(\frac{t-s}{\eta}+1\right)e^{-a_L(t-s)}\\
&\leq \frac{2\Gamma_M\epsilon}{a_L}e^{-\frac{1}{2}a_L(t-s)}
+\frac{\Gamma_M\epsilon}{1+\gamma_L}\left(1+\frac{2}{a_L\eta}\right)e^{-\frac{1}{2}a_L(t-s)}\\
&\leq \epsilon Me^{-\frac{1}{2}a_L(t-s)} .
\end{align*}
The proof is completed.
\end{proof}

It is worth noting that, the proof of Lemma \ref{lem7} is different from those
in \cite{ANS,SP,Stamov}. By employing Lemma \ref{lem4}, we obtain
a new bound for constant $M$ given in \eqref{eq4}.

\begin{lemma}[\cite{L}]\label{lem8} 
Assume that there exist constants $R,S>0,\tau\geq 0$, $T_0\in\mathbb{R}$ and a function 
$y\in PLC([T_0-\tau,\infty),\mathbb{R}^+)$ satisfying
\begin{itemize}
\item[\rm(d1)] $\Delta y(t_k)\leq \gamma_ky(t^-_k)$ for $t_k\geq T_0$, where $\gamma_k> -1$ and 
$\max_{ t_k\geq T_0}\left\{(\gamma_k+1)^{-1},1\right\} <\dfrac{R}{S}$;
\item[\rm(d2)] $D^+y(t)\leq -Ry(t)+S\overline{y}(t)$ for $t\geq T_0, t\neq t_k$, 
where $\overline{y}(t)=\sup_{t-\tau \leq s \leq t}y(s)$ and 
$D^+$ denotes the upper-right Dini derivative;
\item[\rm(d3)] $\tau\leq t_k-t_{k-1}$ for  all $k\in \mathbb{Z}$ satisfies $t_k\geq T_0$.
\end{itemize}
Then
\begin{equation*}
y(t)\leq\overline{y}(T_0)\bigg(\prod_{T_0< t_k \leq t}(\gamma_k+1)\bigg)e^{-\lambda (t-T_0)}, \forall t\geq T_0,
\end{equation*}
where $0<\lambda \leq R-S \max_{ t_k\geq T_0}\left\{(\gamma_k+1)^{-1},1\right\}e^{\lambda \tau}.$
\end{lemma}

\section{Main results} 

Let us set $\mathcal{D}_1=\Big\{\phi\in PLC(\mathbb{R},\mathbb{R}): \phi \text{ is almost periodic}, 
\phi(t)\geq 0 \text{ for all } t\in \mathbb{R}\Big\}$ and $\|\phi\|=\sup _{t\in \mathbb{R}}|\phi (t)|$. 
We define an operator $F: \mathcal{D}_1\to PLC(\mathbb{R},\mathbb{R})$ as follows
\begin{equation}\label{eq14}
\begin{aligned}
F\phi(t)=\int_{-\infty}^tH(t,s)\sum_{i=1}^m\bigg\{ &\frac{b_i(s)}{1+\phi^{\alpha_i}(s-\tau_i(s))}
+c_i(s)\int_0^T\frac{v_i(r)}{1+\phi^{\beta_i}(s-r)}dr\\
&-H_i(s,\phi(s-\sigma_i(s)))\bigg\}ds+\sum_{t_k<t}H(t,t_k)\delta_k.
\end{aligned}
\end{equation}

It can be verified that $x^*(t)=\phi (t)$ is an almost periodic solution on 
$\mathcal{D} _1$ of \eqref{eq6} if and only if $F\phi = \phi$. 

We define the following constants
\begin{equation}\label{eq15}
\begin{aligned}
\underline{\delta}&= \inf_{k\in\mathbb{Z}}|\delta_k|, \ \overline{\delta}
= \sup_{k\in\mathbb{Z}}|\delta_k|, \ 
\overline{\eta} =\sup_{k\in \mathbb{Z}}t^1_k,\\
M_1&=\frac{A}{a_L}\sum_{i=1}^m\left(b_{iM}+c_{iM}-H_{iL}\right)+\frac{A\delta_M}{1-e^{-a_L\eta}},\\
M_2&= \begin{cases}\displaystyle
\frac{B}{a_M}\sum_{i=1}^m\bigg(\frac{b_{iL}}{1+M_1^{\alpha_i}}
+\frac{c_{iL}}{1+M_1^{\beta_i}}-H_{iM}\bigg)
+\frac{B\delta_Le^{-a_M{\overline{\eta}}}}{1-e^{-a_M{ \bar \eta}}}\ \text{ if } \delta_L\geq 0,\\
\frac{B}{a_M}\sum_{i=1}^m\bigg(\dfrac{b_{iL}}{1+M_1^{\alpha_i}}
+\frac{c_{iL}}{1+M_1^{\beta_i}}-H_{iM}\bigg)
+\frac{A\delta_L}{1-e^{-a_L \eta}}\ \text{ if } \delta_L<0.
\end{cases}
\end{aligned}
\end{equation}

\begin{lemma}\label{lem9} 
Let Assumptions {\rm(A1)-(A7)} hold. If $\phi\in\mathcal{D}_1$ then
\[
M_2\leq F\phi (t)\leq M_1,\ \forall t\in \mathbb{R}.
\]
\end{lemma}

\begin{proof} Let $\phi \in\mathcal{D}_1$. By Lemma \ref{lem5} and Lemma \ref{lem6}, from \eqref{eq14} we have 
\begin{equation}\label{eq16}
\begin{aligned}
F\phi (t)&\leq\int_{-\infty}^tAe^{-a_L(t-s)}\sum_{i=1}^m\left(b_{iM}+c_{iM}-H_{iL}\right)ds
+A\delta_M\sum_{t_k<t}e^{-a_L(t-t_k)}\\
&\leq \frac{A}{a_L}\sum_{i=1}^m\left(b_{iM}+c_{iM}-H_{iL}\right)
+\frac{A\delta_M}{1-e^{-a_L\eta}}=M_1,\ \forall t\in \mathbb{R}.
\end{aligned}
\end{equation}

For each $t\in \mathbb{R}$, let $n_0$ be an integer such that 
$t_{n_0}<t\leq t_{n_0+1}$. If $\delta_L\geq 0$ then, by Lemma \ref{lem6},
it follows from the fact 
$(n-k)\eta \leq t_n-t_k\leq (n-k)\bar \eta,\ \forall k\leq n$ that
\begin{equation}\label{eq17}
\begin{aligned}
\sum_{t_k<t}H(t,t_k)\delta_k &\geq \sum_{t_k<t}B\delta_L e^{-a_M (t-t_k)}
\geq \sum_{t_k\leq t_{n_0}}B\delta_Le^{-a_M (t_{n_0+1}-t_k)}\\
&=\sum_{k\leq n_0}B\delta_Le^{-a_M (t_{n_0+1}-t_k)}
\geq\sum_{q=1}^{\infty}B\delta_Le^{-a_M \bar \eta q}
=\frac{B\delta_Le^{-a_M\bar  \eta}}{1-e^{-a_M\bar  \eta}}.
\end{aligned}
\end{equation}

If $\delta_L<0$ then from (A7) and Lemma \ref{lem6}, we have
\begin{equation}\label{eq18}
\begin{aligned}
\sum_{t_k<t}H(t,t_k)\delta_k&\geq \sum_{t_k<t}A\delta_L e^{-a_L(t-t_k)}
\geq \sum_{t_k\leq t_{n_0}}A\delta_L e^{-a_L(t_{n_0}-t_k)}\\
&\geq \sum_{q=0}^\infty A\delta_Le^{-a_Lq\eta}=\frac{A\delta}{1-e^{-a_L\eta}}.
\end{aligned}
\end{equation}

From \eqref{eq17} and \eqref{eq18} we obtain
\begin{equation}\label{eq19}
\begin{aligned}
F\phi(t)\geq \frac{B}{a_M}\sum_{i=1}^m\bigg(\frac{b_{iL}}{1+M_1^{\alpha_i}}
+\frac{c_{iL}}{1+M_1^{\beta_i}}-H_{iM}\bigg)+\sum_{t_k<t}H(t,t_k)\delta_k\geq M_2.
\end{aligned}
\end{equation}
The proof is completed.
\end{proof}

Now we are in position to introduce our main results as follows.

\begin{theorem}\label{thm2}  
Under the Assumptions {\rm(A1)-(A7)}, if $\phi\in \mathcal{D}_1$
then $F\phi (t)$ is almost periodic.
\end{theorem}

\begin{proof} 
Let $\phi \in\mathcal{D}_1$. For given 
$\epsilon\in(0, \eta/2)$, there exists $0<\delta <\epsilon/2$ 
such that, if $t,\bar t$ belong to the same interval of continuity of $\phi(t)$ then 
\begin{equation}\label{eq20}
\left|\phi(t)-\phi\left(\bar t\right)\right|<\epsilon, \; |t-\bar t|<\delta.
\end{equation}

By Lemma \ref{lem2} and Lemma \ref{lem7}, there exist $0< \epsilon _1<\delta$, 
relatively dense sets $\Omega\subset\mathbb{R}$, $\mathcal{P}\subset\mathbb{Z}$ such that,
for all $\omega\in\Omega$, we have
\begin{equation}\label{eq21}
\begin{aligned}
&|H(t+\omega, s+\omega)-H(t,s)|\leq \delta Me^{-\frac{1}{2}a_L(t-s)},\ 
t\geq s, |t-t_k|>\delta, |s-t_k| >\delta;\\
& |\phi (t+\omega)-\phi (t)|<\delta, \ t\in \mathbb{R}, |t-t_k|>\delta, k\in \mathbb{Z};\\
&|H_i(t+\omega,x)-H_i(t,x)|<\delta, \ t\in \mathbb{R}, x\in [\phi_L,\|\phi\|], i\in\underline{m};\\
&|a(t+\omega)-a(t)|<\delta, \ t\in \mathbb{R};\\
&|b_i(t+\omega)-b_i(t)|<\delta,\; |c_i(t+\omega)-c_i(t)|<\delta, \ t\in \mathbb{R}, i\in\underline{m};\\
& |\tau_i(t+\omega)-\tau_i(t)|<\delta, \; |\sigma_i(t+\omega)-\sigma_i(t)|<\delta, \ t\in \mathbb{R}, i\in\underline{m};\\
&|\gamma _{k+p}-\gamma _k|< \delta,\; |\delta _{k+p}-\delta _k|< \delta,\; 
|t^p_k-\omega|<\epsilon _1, p\in\mathcal{P}, k\in \mathbb{Z}.
\end{aligned}
\end{equation}

Let $\omega \in \Omega$, $p\in \mathcal{P}$. One can easily see that
\begin{equation}\label{eq22}
\begin{aligned}
F\phi (t+\omega)=\; &\int_{-\infty}^tH(t+\omega,s+\omega)
\sum_{i=1}^m\bigg\{\frac{b_i(s+\omega)}{1+\phi^{\alpha_i}(s+\omega-\tau_i(s+\omega))}\\
&+\int_0^T\frac{c_i(s+\omega)v_i(r)}{1+\phi^{\beta_i}(s+\omega-r)}dr
-H_i(s+\omega,\phi(s+\omega-\sigma_i(s+\omega)))\bigg\}ds\\
&+\sum_{t_k<t}H(t+\omega,t_{k+p})\delta_{k+p}.
\end{aligned}
\end{equation}

We define $E_\epsilon(\{t_k\})=\{t\in \mathbb{R}:\; |t-t_k|>\epsilon, \forall k\in \mathbb{Z}\}$. 
For $t\in E_\epsilon(\{t_k\})$, $i\in\underline{m}$, let us set
\begin{equation}\label{eq23}
\begin{aligned}
C_i&=\int_{-\infty}^t\left|\frac{H(t+\omega,s+\omega)
b_i(s+\omega)}{1+\phi^{\alpha_i}(s+\omega-\tau_i(s+\omega))}
-\frac{H(t,s)b_i(s)}{1+\phi^{\alpha_i}(s-\tau_i(s))}\right|ds,\\
D_i&=\int_{-\infty}^t\left|H(t+\omega,s+\omega)\int_0^T
\frac{c_i(s+\omega)v_i(r)}{1+\phi^{\beta_i}(s+\omega-r)}dr
-H(t,s)\int_0^T\frac{c_i(s)v_i(r)}{1+\phi^{\beta_i}(s-r)}dr\right|ds,\\
E_i&=\int_{-\infty}^t\left|H(t+\omega,s+\omega)H_i(s+\omega,\phi(s+\omega-\sigma_i(s+\omega)))
-H(t,s)H_i(s,\phi(s-\sigma_i(s)))\right|ds,\\
G&=\sum_{t_k<t}\left|H(t+\omega,t_{k+p})\delta_{k+p}-H(t,t_k)\delta_k\right|,
\end{aligned}
\end{equation}
then we have 
\begin{equation}\label{eq24}
|F\phi(t+\omega)-F\phi(t)|\leq \sum_{i=1}^m(C_i+D_i+E_i)+G,\; t\in E_\epsilon(\{t_k\}).
\end{equation}

We also define 
\begin{equation}\label{eq25}
\begin{aligned}
C_{i1}&=\int_{-\infty}^t|H(t+\omega,s+\omega)-H(t,s)|
\frac{|b_i(s+\omega)|}{1+\phi^{\alpha_i}(s+\omega-\tau_i(s+\omega))}ds,\\
C_{i2}&=\int_{-\infty}^tH(t,s)
\frac{|b_i(s+\omega)-b_i(s)|}{1+\phi^{\alpha_i}(s+\omega-\tau_i(s+\omega))}ds,\\
C_{i3}&=\int_{-\infty}^tH(t,s)|\phi(s+\omega-\tau_i(s+\omega))-\phi(s-\tau_i(s+\omega))|ds,\\
C_{i4}&=\int_{-\infty}^tH(t,s)|\phi(s-\tau_i(s+\omega))-\phi(s-\tau_i(s))|ds
\end{aligned}
\end{equation}
and $K_i=\sup_{\phi_L\leq x\leq \|\phi\|}\alpha_i x^{\alpha_i-1}$. 
It can be seen from \eqref{eq23} and \eqref{eq25} that
\begin{equation}\label{eq26}
C_i\leq C_{i1}+C_{i2}+b_{iM}K_i(C_{i3}+C_{i4}),\; i\in\underline{m}.
\end{equation}

By Lemma \ref{lem5} and Lemma \ref{lem6}, from \eqref{eq21}, \eqref{eq25} 
and the fact that $\int_{-\infty}^te^{-\frac{1}{2}a_L(t-s)}ds=2/a_L$, we have
\begin{equation}\label{eq27}
C_{i1}\leq \frac{2b_{iM}M}{a_L}+\sum_{t_k<t}2Ab_{iM}
\int_{t_k-\epsilon}^{t_k+\epsilon}e^{-a_L(t-s)}ds\leq b_{iM}
\bigg(\frac{2M}{a_L}+\frac{4Ae^{\frac{1}{2}a_L\eta}}{1-e^{-a_L\eta}}\bigg)\epsilon,
\end{equation}
and 
\begin{equation}\label{eq28}
\begin{aligned}
&C_{i2}\leq \int_{-\infty}^tA\epsilon e^{-a_L(t-s)}ds=\frac{A}{a_L}\epsilon,\\
&C_{i3}\leq \int_{-\infty}^t A\epsilon e^{-a_L(t-s)}ds+2A\|\phi\|\sum_{t_k<t}
\int_{\{s: |s-\tau_i(s+\omega)-t_k|<\epsilon, s\leq t\}}e^{-a_L(t-s)}ds.
\end{aligned}
\end{equation}

It should be noted that, by (A4), $t-\tau_i(t)$, $i\in\underline{m}$,
are strictly increasing functions, and thus, 
there exist the inverse functions $\tau^*_i(t)$ of $t-\tau_i(t)$.
For each $t\in\mathbb{R}$, denote $\bar t=t-\epsilon-\tau_i(t+\omega)$ then 
\begin{equation}\label{eq29}
t+\omega =\tau^*_i(\bar t+\omega +\epsilon).
\end{equation}

Let $\underline{\lambda}_i=\inf_{s\in\mathbb{R}}\dot\tau^*_i(s)$, 
$\overline{\lambda}_i= \sup_{s\in\mathbb{R}}\dot\tau^*_i(s), i\in\underline{m}$, then, by (A4), 
$0<\underline{\lambda}_i, \overline{\lambda}_i<\infty.$
Therefore, 
\begin{equation}\label{eq30}
\tau^*_i(\bar t+\omega+\epsilon)-\tau^*_i(t_{k}+\omega+\epsilon)
\geq\underline{\lambda}_i(\bar t-t_k),\ t_k<\bar t,
\end{equation}
and hence, from equations \eqref{eq28}-\eqref{eq30}, we have
\begin{equation}\label{eq31}
\begin{aligned}
C_{i3}&\leq \frac{A\epsilon}{a_L}+2A\|\phi\|\sum_{t_k<\bar t}
\int_{\tau^*_i(t_k+\omega-\epsilon)-\omega}^{\tau^*_i(t_k+\omega+\epsilon)-\omega}e^{-a_L(t-s)}ds\\
& \leq \frac{A\epsilon}{a_L}+2A\|\phi\|\sum_{t_k<\bar t}
e^{-a_L[\tau^*_i(\bar t+\omega+\epsilon)-\tau^*_i(t_k+\omega+\epsilon)]}
\left[\tau^*_i( t_k+\omega+\epsilon)-\tau^*_i(t_k+\omega-\epsilon)\right]\\
&\leq\frac{A\epsilon}{a_L}+4A\|\phi\|\overline{\lambda}_i\epsilon
\sum_{t_k<\bar t}e^{-a_L{\underline \lambda}_i(\bar t-t_k)}
\leq \bigg(\frac{1}{a_L}+\frac{4\overline{\lambda}_i\|\phi\|}{1-e^{-a_L{\underline\lambda}_i\eta}}\bigg)A\epsilon.
\end{aligned}
\end{equation}

By the same arguments used in deriving \eqref{eq31}, we obtain
\begin{equation}\label{eq32}
C_{i4}\leq A\epsilon\int_{-\infty}^te^{-a_L(t-s)}ds
+2A\|\phi\|\sum_{t_k<t}\int_{\tau^*_i(t_k-\epsilon)}^{\tau^*_i(t_k+\epsilon)}e^{-a_L(t-s)}ds
 \leq \bigg(\frac{1}{a_L}+\frac{4\overline{\lambda}_i\|\phi\|}{1-e^{-a_L{\underline\lambda}_i\eta}}\bigg)A\epsilon.
\end{equation}
Combining \eqref{eq26}-\eqref{eq28}, \eqref{eq31} and \eqref{eq32}, we readily obtain
\begin{equation}\label{eq33}
C_i\leq \left[\frac{A}{a_L}+\frac{2b_{iM}(M+AK_i)}{a_L}+Ab_{iM}
\bigg(\frac{4e^{\frac{1}{2}a_L\eta}}{1-e^{-a_L\eta}}
+\frac{8K_i\overline{\lambda}_i\|\phi\|}{1-e^{-a_L{\underline \lambda}_i\eta}}\bigg)\right]\epsilon.
\end{equation}

Next, let us set 
\begin{equation}\label{eq34}
\begin{aligned}
D_{i1}&=\int_{-\infty}^t|H(t+\omega,s+\omega)-H(t,s)|
\int_0^T\frac{v_i(r)}{1+\phi^{\beta_i}(s+\omega-r)} drds,\\
D_{i2}&=\int_{-\infty}^tH(t,s)|c_i(s+\omega)-c_i(s)|
\int_0^T\frac{v_i(r)}{1+\phi^{\beta_i}(s+\omega-r)} drds,\\
D_{i3}&=\int_{-\infty}^tH(t,s)\int_0^Tv_i(r)
\left|\frac{1}{1+\phi^{\beta_i}(s+\omega-r)}-\frac{1}{1+\phi^{\beta_i}(s-r)}\right|drds.
\end{aligned}
\end{equation}
It follows from \eqref{eq23} and \eqref{eq34} that
\begin{equation}\label{eq35}
D_i\leq c_{iM}D_{i1}+D_{i2}+c_{iM}D_{i3},\ i\in\underline{m}.
\end{equation}
By Lemmas \ref{lem5} and \ref{lem6}, from \eqref{eq21} we have 
\begin{equation}\label{eq36}
\begin{aligned}
D_{i1}&\leq \int_{-\infty}^tM\epsilon e^{-\frac{1}{2}a_L(t-s)}ds
+\sum_{t_k<t}2A\int_{t_k-\epsilon}^{t_k+\epsilon}e^{-a_L(t-s)}
\leq \bigg(\frac{2M}{a_L}+\frac{4Ae^{\frac{1}{2}a_L\eta}}{1-e^{-a_L\eta}}\bigg)\epsilon,\\
D_{i2}&\leq\int_{-\infty}^tA\epsilon e^{-a_L(t-s)}ds=\frac{A}{a_L}\epsilon,
\end{aligned}
\end{equation}
and 
\begin{equation}\label{eq37}
\begin{aligned}
D_{i3}&\leq\int_{-\infty}^t  Ae^{-a_L(t-s)}\int_0^TG_i v_i(r)|\phi(s+\omega-r)-\phi(s-r)|drds\\
&\leq \int_{-\infty}^tAG_i\epsilon e^{-a_L(t-s)}ds+2AG_i\|\phi\|\int_0^Tv_i(r)
\bigg(\sum_{t_k+r<t}\int_{t_k-r-\epsilon}^{t_k+r+\epsilon}e^{-a_L(t-s)}ds\bigg)dr\\
&\leq\frac{AG_i\epsilon}{a_L}+{4AG_i\|\phi\|\epsilon}\int_0^Tv_i(r)
\bigg(\sum_{t_k+r<t}e^{-a_L(t-t_k-r -\epsilon)}\bigg)dr\\
&\leq\bigg(\frac{1}{a_L}+\frac{4e^{\frac{1}{2}a_L\eta}\|\phi\|}{1-e^{-a_L\eta}}\bigg)AG_i\epsilon,
\end{aligned}
\end{equation}
where $G_i=\sup_{\phi_L\leq x\leq \|\phi\|}\beta_i x^{\beta_i-1}$. 
From \eqref{eq35}-\eqref{eq37}, we readily obtain
\begin{equation}\label{eq38}
D_i\leq\left[\frac{A+2Mc_{iM}+Ac_{iM}G_i}{a_L}
+\frac{4Ac_{iM}e^{\frac{1}{2}a_L\eta}(1+G_i\|\phi\|)}{1-e^{-a_L\eta}}\right]\epsilon.
\end{equation}

Now, we define
\begin{equation}\label{eq39}
\begin{aligned}
E_{i1}&=\int_{-\infty}^t|H(t+\omega,s+\omega)-H(t,s)|H_i(s+\omega,\phi(s+\omega-\sigma_i(s+\omega)))ds, \\
E_{i2}&=\int_{-\infty}^tH(t,s)|H_i(s+\omega,\phi(s+\omega-\sigma_i(s+\omega)))
-H_i(s+\omega,\phi(s-\sigma_i(s+\omega)))|ds, \\
E_{i3}&=\int_{-\infty}^tH(t,s)|H_i(s+\omega,\phi(s-\sigma_i(s+\omega)))
-H_i(s+\omega,\phi(s-\sigma_i(s)))|ds,  \\
E_{i4}&=\int_{-\infty}^tH(t,s)|H_i(s+\omega,\phi(s-\sigma_i(s)))- H_i(s,\phi(s-\sigma_i(s)))|ds,
\end{aligned}
\end{equation}
then, from \eqref{eq23} and \eqref{eq39}, we have 
\begin{equation}\label{eq40}
E_i\leq E_{i1}+E_{i2}+E_{i3}+E_{i4},\ i\in\underline{m}.
\end{equation}
Also using Lemma \ref{lem5} and Lemma \ref{lem6}, from \eqref{eq21} and   
the fact that $\int_{-\infty}^te^{-\frac{1}{2}a_L(t-s)}ds=2/a_L$, we obtain
\begin{equation}\label{eq41}
\begin{aligned}
E_{i1}&\leq \frac{2H_{iM}M\epsilon}{a_L}+\sum_{t_k<t}2AH_{iM}
\int_{t_k-\epsilon}^{t_k+\epsilon}e^{-a_L(t-s)}ds
\leq 2\left(\frac{M}{a_L}+\frac{2Ae^{\frac{1}{2}a_L\eta}}{1-e^{-a_L\eta}}\right)H_{iM}\epsilon,\\
E_{i4}&\leq \int_{-\infty}^tA\epsilon e^{-a_L(t-s)}ds \leq \frac{A}{a_L}\epsilon.
\end{aligned}
\end{equation}
Let $\underline{\xi}_i=\inf_{t\in \mathbb{R}}\dot \sigma^*_i(t),\overline{\xi}_i
=\sup_{t\in \mathbb{R}}\dot \sigma^*_i(t), i\in\underline{m}$. 
Similarly to \eqref{eq31} and \eqref{eq32}, we readily obtain
\begin{equation}\label{eq42}
\begin{aligned}
E_{i2}&\leq L_i\int_{-\infty}^tH(t,s)|\phi(s+\omega-\sigma_i(s+\omega))-\phi(s-\sigma_i(s+\omega))|ds\\
&\leq \left(\frac{1}{a_L}+\frac{4\overline{\xi}_i\|\phi\|}{1-e^{-a_L{\underline\xi}_i\eta}}\right) AL_i\epsilon,\\
 E_{i3}&\leq L_i\int_{-\infty}^tH(t,s)|\phi(s-\sigma_i(s+\omega))-\phi(s-\sigma_i(s))|ds\\
&\leq \left(\frac{1}{a_L}+\frac{4\overline{\xi}_i\|\phi\|}{1-e^{-a_L{\underline\xi}_i\eta}}\right) AL_i\epsilon.
\end{aligned}
\end{equation}
Inequalities \eqref{eq40}-\eqref{eq42} yield 
\begin{equation}\label{eq43}
E_i\leq\bigg(\frac{A(2L_i+1)+2H_{iM}M}{a_L}
+\frac{4Ae^{\frac{1}{2}a_L\eta}  H_{iM}}{1-e^{-a_L\eta}}
+\frac{8AL_i\overline{\xi}_i\|\phi\|}{1-e^{-a_L\underline{\xi}_i\eta}}\bigg)\epsilon.
\end{equation}

Let us set 
\begin{equation}\label{eq44}
G_1= \sum_{t_k<t}|H(t+\omega,t_{k+p})-H(t,t_k)|, \quad G_2= \sum_{t_k<t}| H(t,t_k)|
\end{equation}
then 
\begin{equation}\label{eq45}
G\leq\sum_{t_k<t}|H(t+\omega,t_{k+p})-H(t,t_k)||\delta_{k+p}|
+\sum_{t_k<t}| H(t,t_k)||\delta_{k+p}-\delta_k|
\leq \bar\delta G_1+\epsilon G_2.
\end{equation}
For each $t\geq t_k$, there exists a unique integer $l=l(t)$ such that $t_l<t\leq t_{l+1}$. 
By Lemma \ref{lem3} we have $t_{l+p}<t+\omega<t_{l+p+1}$. Thus 
\begin{equation}\label{eq46}
\begin{aligned}
 G_1&\leq\sum_{t_k<t}\bigg(\Gamma_M e^{-a_L(t-t_k-\epsilon)}
\left|\int_{t_{k+p}}^{t+\omega}a(s)ds-\int_{t_{k}}^{t}a(s)ds\right|\\
&\quad +e^{-a_L(t-t_k)}\left|\Gamma (k+p,l+p)-\Gamma (k,l)\right|\bigg)\\
&\leq \Gamma_M I_1 +I_2,
\end{aligned}
\end{equation}
where 
\begin{equation}\label{eq47}
\begin{aligned}
&I_1=\sum_{t_k<t} e^{-a_L(t-t_k-\epsilon)}
\left|\int_{t_{k+p}}^{t+\omega}a(s)ds-\int_{t_{k}}^{t}a(s)ds\right|,\\
&I_2= \sum_{t_k<t}e^{-a_L(t-t_k)}\left|\Gamma (k+p,l+p)-\Gamma (k,l)\right|.
\end{aligned}
\end{equation}
Note that $\epsilon <\eta/2, \ \dfrac12a_L(t-t_k)e^{-\frac{1}{2}a_L(t-t_k)}<1$ and 
$|t^p_k-\omega |<\epsilon_1<\epsilon$, from \eqref{eq21} and Lemma \ref{lem5}, we have 
\begin{equation}\label{eq48}
\begin{aligned}
I_1&\leq \sum_{t_k<t} e^{-a_L(t-t_k-\epsilon)}\bigg(\int_{t_{k}}^{t}|a(s+\omega)-a(s)|ds
+\bigg|\int_{t_{k+p}-\omega}^{t_k}a(s)ds\bigg|\bigg)\\
&\leq \epsilon\sum_{t_k<t} e^{-a_L(t-t_k-\epsilon)}(t-t_k+a_M)\\
&\leq\left(a_M\epsilon e^{\frac{1}{2}a_L\eta}
+\frac{2\epsilon}{a_L}e^{\frac{1}{2}a_L\eta}\right)
\sum_{t_k<t} e^{-\frac{1}{2}a_L(t-t_k)}\\
&\leq\frac{e^{\frac{1}{2}a_L\eta}}{1-e^{-\frac{1}{2}a_L\eta}}
\bigg(a_M+\frac{2 }{a_L}\bigg)\epsilon.
\end{aligned}
\end{equation}
Similarly, we obtain 
\begin{equation}\label{eq49}
\begin{aligned}
I_2&\leq \frac{\Gamma_M}{1+\gamma_L}
\sum_{t_k<t}e^{-a_L(t-t_k)}(l-k+1)\epsilon
\leq \frac{\Gamma_M}{1+\gamma_L}\sum_{t_k<t}e^{-a_L(t-t_k)}
\bigg(\frac{t-t_k}{\eta}+1\bigg)\epsilon\\
 &\leq\frac{\Gamma_M}{(1+\gamma_L)(1-e^{-\frac{1}{2}a_L\eta})}
\bigg(\frac{2}{a_L\eta }+1\bigg)\epsilon.
\end{aligned}
\end{equation}
It follows from \eqref{eq46}-\eqref{eq49} that
\begin{equation}\label{eq50}
G_1\leq\frac{\Gamma_M}{1-e^{-\frac{1}{2}a_L\eta}}
\left[e^{\frac{1}{2}a_L\eta}\left(a_M+\frac{2}{a_L}\right)
+\frac{1}{1+\gamma_L}\left(\frac{2}{a_L\eta}+1\right)\right]\epsilon.
\end{equation}
We also have
\begin{equation}\label{eq51}
G_2 \leq A\sum_{t_k<t}e^{-a_L(t-t_k)}\leq \frac{A}{1-e^{-a_L\eta}}.
\end{equation}
Therefore
\begin{equation}\label{eq52}
G\leq\frac{\epsilon\overline\delta\Gamma_M}{1-e^{-\frac{1}{2}a_L\eta}}
\left[e^{\frac{1}{2}a_L\eta}\bigg(a_M+\frac{2 }{a_L}\bigg)
+\frac{1}{1+\gamma_L}\bigg(\frac{2}{a_L\eta }+1\bigg)\right]
+\frac{A\epsilon}{1-e^{-\frac{1}{2}a_L\eta}}.
\end{equation}
We can see clearly from \eqref{eq24}, \eqref{eq33}, 
\eqref{eq38}, \eqref{eq43} and \eqref{eq52} that 
there exists a positive constant $\Lambda$ such that 
$|F(\phi (t+\omega)-F\phi (t)|\leq \Lambda \epsilon$, for all 
$t\in\mathbb{R}, |t-t_k|>\epsilon$, $k\in \mathbb{Z}$. 
This shows that $F\phi(t)$ is almost periodic. The proof is completed.
\end{proof} 

\begin{theorem}\label{thm3}  
Let Assumptions {\rm (A1)-(A7)} hold. If  $M_2$, defined in \eqref{eq15}, is positive and 
\begin{equation}\label{eq53}
\frac{A}{a_L}\sum_{i=1}^m\left(b_{iM}K^*_i+c_{iM}G^*_i+L_i\right)<1,
\end{equation}
where $K^*_i=\sup_{M_2\leq x\leq M_1}\alpha_ix^{\alpha_i-1},\ G^*_i
=\sup_{M_2\leq x\leq M_1}\beta_ix^{\beta_i-1}$, 
then equation \eqref{eq6} has a unique positive almost periodic solution.
\end{theorem} 

\begin{proof} We define 
$\mathcal{D}_2=\big\{\phi \in\mathcal{D}_1: M_2
\leq \phi(t)\leq M_1,\ t\in \mathbb{R}\big\}$. 
It is worth noting that, from Lemma \ref{lem9}, 
Theorem \ref{thm2} and the assumption $M_2>0$, we have 
$F(\mathcal{D}_2)\subset {\mathcal D}_2$. 
For any $\phi, \psi\in\mathcal{D}_2$, applying Lemma \ref{lem6} we obtain
\begin{align*}
|F\phi (t)-F\psi (t)|\leq\; &\int_{-\infty}^tH(t,s)\sum_{i=1}^m\bigg\{L_i|\phi(s-\sigma_i(s))-\psi(s-\sigma_i(s))|\\
&+c_{iM}\int_0^Tv_i(r)\left|\frac{1}{1+\phi^{\beta_i}(s-r)}-\frac{1}{1+\psi^{\beta_i}(s-r)}\right|dr \\
&+ b_{iM}\left|\frac{1}{1+\phi^{\alpha_i}(s-\tau_i(s))}-\frac{1}{1+\psi^{\alpha_i}(s-\tau_i(s))}\right|\bigg\}ds\\
\leq \; & \frac{A}{a_L}\sum_{i=1}^m\left(b_{iM}K^*_i+c_{iM}G^*_i+L_i\right)\|\phi -\psi\|.
\end{align*}
Therefore
\begin{equation*}
\|F\phi -F\psi \|\leq \frac{A}{a_L}\sum_{i=1}^m\left(b_{iM}K^*_i+c_{iM}G^*_i+L_i\right)\|\phi-\psi\|
\end{equation*}
which yields $F$ is a contraction mapping on ${\mathcal D}_2$ by condition \eqref{eq53}. 
Then $F$ has a unique fixed point in $\mathcal{D}_2$, namely $\phi_0$. It should be noted that, 
$F(\mathcal{D}_1)\subset\mathcal{D}_2$, and hence, 
$F$ also has a unique fixed point $\phi_0$ in $\mathcal{D}_1$. 
This shows that \eqref{eq6} has a unique positive almost periodic solution 
$x^*(t)=\phi_0(t)$. The proof is completed.
\end{proof}

\begin{theorem}\label{thm4} 
Let Assumptions {\rm(A1)-(A7)} hold. If $M_2>0, \sigma\leq\eta$ and 
\begin{equation}\label{eq54}
\frac{1}{a_L}\max\left\{A, (\gamma _L +1)^{-1}\right\}
\sum_{i=1}^m\left(b_{iM}K^*_i+c_{iM}G^*_i+L_i\right)<1,
\end{equation}
then \eqref{eq6} has a unique positive almost periodic solution $x^*(t)$. 
Moreover, every solution $x(t)=x(t,\alpha,\xi)$ of \eqref{eq6} converges exponentially to 
$x^*(t)$ as $t\to \infty$.
\end{theorem} 

\begin{proof} By Theorem \ref{thm3}, \eqref{eq6} 
has a unique positive almost periodic solution $x^*(t)$. 
Let $x(t)=x(t,\alpha,\xi)$ be a solution of \eqref{eq6} and \eqref{eq9}.  
We define $V(t)=|x(t)-x^*(t)|$ then 
\begin{align*}
D^+V(t) & \leq -a_L|x(t)-x^*(t)|+\sum_{i=1}^m
\bigg[b_{iM}K^*_i|x(t-\tau_i(t))-  x^*(t-\tau_i(t))| \\
&\quad+c_{iM}G^*_i\int_0^Tv_i(s)|x(t-s)-x^*(t-s)|ds+L_i|x(t-\sigma_i(t))-  x^*(t-\sigma_i(t))|\bigg]\\
 &\leq -a_LV(t)+ \sum_{i=1}^m\left(b_{iM}K^*_i+c_{iM}G^*_i +L_i\right)\overline{V}(t), \; t\ne t_k,\ t\geq \alpha,\\
\Delta V(t_k)&=\gamma_k V(t^-_k), \; t_k\geq\alpha, k\in \mathbb{Z},
\end{align*}
where $\overline{V}(t)=\sup_{t-\sigma\leq s \leq t}V(s)$.
By Lemma \ref{lem8}, there exists a positive constant $\lambda$ such that 
\begin{equation*}
V(t)\leq \overline{V}(\alpha)\prod_{\alpha <t_k\leq t}(\gamma_k+1)e^{-\lambda (t-\alpha)}
\leq\Gamma_M\overline{V}(\alpha) e^{-\lambda (t-\alpha)},\; t\geq \alpha.
\end{equation*}
This shows that $x(t)$ converges exponentially to $x^*(t)$ as $t\to \infty$. 
The proof is completed.
\end{proof}

The existence and exponential stability of positive almost periodic solution of \eqref{eq4} 
is presented in the following corollary 
as an application of our obtained results with $m=1$.

\begin{corollary}\label{cr1} Under the Assumptions {\rm (A1), (A2)} (with $c(t)=0$) and {\rm (A5)-(A7)}, 
if $M_2>0, \tau\leq\eta$ and 
\begin{equation}\label{eq55}
\frac{1}{a_L}\max\left\{A, (\gamma _L +1)^{-1}\right\}b_MK^*<1,
\end{equation}
where $K^*=\sup_{M_2\leq x\leq M_1}\alpha x^{\alpha-1}$, then \eqref{eq4} has a 
unique positive almost periodic solution which is exponentially stable. 
\end{corollary}

\section{An illustrative example}

In this section we give a numerical example to illustrate the effectiveness 
of our conditions. For illustrating purpose, let us consider the following equation
\begin{equation}\label{eq56}
\begin{split}
\dot x(t)=\; &-a(t)x(t)+\frac{b(t)}{1+x^2(t-\tau(t))}+\int_0^1\frac{c(t)ds}{1+x^2(t-s)}\\
&-d(t)\frac{|x(t-\sigma(t))|}{10+|x(t-\sigma(t))|},\quad t\neq k,\\
\Delta x(k)=\;&\gamma_kx(k-0)+\delta_k, \quad k\in \mathbb{Z},
\end{split}
\end{equation}
where
\begin{align*}
&a(t)=5+|\sin(t\sqrt 2)|,\; b(t)=\frac1{10}(1+|\sin(t\sqrt 3)|),\; 
c(t)=\frac1{10}(1+|\cos(t\sqrt 3)|),\\
&d(t)=\frac1{20}\sin^2(t\sqrt3),\; \tau(t)=\sin^2(\frac{\sqrt3}{2}t),\; 
\sigma(t)=\cos^2(\frac{\sqrt3}{2}t),\\
&\gamma_{2m}=-\frac12,\ \gamma_{2m+1}=1,\ \delta_{2m}=1,\ 
\delta_{2m+1}= \frac12,\ m\in\mathbb{Z}.
\end{align*}

It should be noted that, the functions $a(t), b(t), c(t), \tau(t)$ and $\sigma(t)$ 
are almost periodic in the sense of Bohr, 
$H(t,x)=d(t)\dfrac{|x|}{10+|x|}$ is almost periodic in $t\in\mathbb{R}$ 
uniformly in $x\in\mathbb{R}_+$, 
$|H(t,x)-H(t,y)|\leq \frac12|x-y|$. Therefore, Assumptions (A1)-(A5) and (A7) are satisfied. 
On the other hand, $\Gamma(q,p)=\prod_{i=q}^p(1+\gamma_i)\in\{\frac12, 1, 2\}$ 
for any $p,q\in\mathbb{Z}, p\geq q$. 
Thus, $\Gamma_M=2, \Gamma_L=\dfrac12$ and Assumption (A6) is satisfied. 
Taking some computations we obtain
\begin{align*}
&a_L=5,\ a_M=6,\ b_L=c_L=0.1,\ b_M=c_M=0.2, L\leq 0.5, H_M=\frac1{20}, H_L=0,\\
&A=2,\ B=0.5,\ M_1=2.1736, \ M_2=0.0027,\ K^*=G^*=2M_1
\end{align*}
and $\dfrac{A}{a_L}(b_MK^*+c_MG^*+L)\leq 0.8956$. 
By Theorem \ref{thm3}, equation \eqref{eq56} 
has a unique positive almost periodic solution $x^*(t)$. 
Furthermore, it can be seen that 
$\gamma_L=-\frac12$, and hence, 
$\dfrac1{a_L}\max\{A,(\gamma_L+1)^{-1}\}(b_MK^*+c_MG^*+L)\leq 0.8956$. 
By Theorem \ref{thm4}, every solution $x(t,\alpha,\xi)$ of \eqref{eq56} 
converges exponentially to $x^*(t)$ as $t$ tends to 
infinity. As presented in figure 1, state trajectories of \eqref{eq56} 
with different initial conditions converge 
to the unique positive almost periodic solution of \eqref{eq56}.

  \begin{figure}[!ht]\label{fig1}
  \begin{center}
  \includegraphics[width=0.6\textwidth]{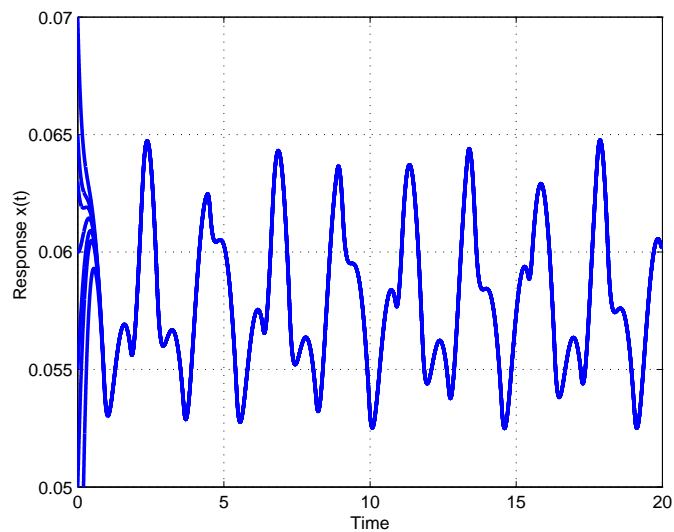}
\caption{State trajectories of \eqref{eq56} converge to the unique 
positive almost periodic solution}
  \end{center}
  \end{figure}

\section{Conclusion}
This paper has dealt with the existence and exponential attractivity 
of a unique positive almost periodic solution for a 
generalized model of hematopoiesis with delays and impulses. 
Using the contraction mapping principle and 
a novel type of impulsive delay inequality, new sufficient conditions have 
been derived ensuring that all solutions of the model 
converge exponentially to the unique positive almost periodic solution.


\begin{thebibliography}{00}

\bibitem{ANS} J.O. Alzabut, J.J. Nieto and G.T. Stamov, Existence and exponential stability of  positive almost periodic solution 
for a model of hematopoiesis, \emph{Bound. Value Prob.}, {\bf 10}(2009), Article ID 127510.

\bibitem{A} T. T. Anh, Existence and global asymptotic stability of positive periodic solutions of a Lotka-Volterra type 
competition systems with delays and feedback controls, 
\emph{Electron. J. Diff. Eq.}, Vol. 2013 (2013), No. 261,  1--16. 

\bibitem{BS} D. D. Bainov and P.S. Simeonov, \emph{Impulsive differential equations: Periodic solutions and applications}, 
Longman Scientific, New York, 1993.

\bibitem{BBI1} L. Berezansky, E. Braverman and L. Idels, Nicholson's blowflies differential equations revisited:
Main results and open problems, \emph{Appl. Math. Modelling}, {\bf 34}(2010), 1405--1417.

\bibitem{BBI2} L. Berezansky, E. Braverman and L. Idels, Mackey-Glass model of hematopoiesis with non-monotone 
feedback: Stability, oscillation and control, \emph{Appl. Math. Comput.}, {\bf219}(2013), 6268--6283.

\bibitem{F} A. M. Fink, \emph{Almost Periodic Differential Equations}, Springer-Verlag, New York, 1974.

\bibitem{GL} I. Gyori and G. Ladas, \emph{Oscillation Theory of Delay Differential Equations with Applications}, 
Clarendon, Oxford, 1991.

\bibitem{HW} P. Hong and P. Weng, Global attractivity of almost-periodic solution in a model
of hematopoiesis with feedback control, \emph{Nonlinear Anal. RWA}, {\bf 12}(2011), 2267--2285.

\bibitem{HLSC}  M. Huang, S. Liu, X. Song and L. Chen, Periodic solutions and homoclinic bifurcation of a predator-prey 
system with two types of harvesting, \emph{Nonlinear Dyn.}, {\bf73}(2013), 815--826.

\bibitem{LZ} B. M. Levitan and V. V. Zhikov, \emph{Almost Periodic Functions and Differential Equations},
 Cambridge University Press, Cambridge, 1983.

\bibitem{L}  X. Li, Global exponential stability of delay neural networks with impulsive perturbations,
\emph{Adv. Dyn. Syst. Appl.}, {\bf5}(2010), 107--122.

\bibitem{LM} X. Liu and J. Meng, The positive almost periodic solution for Nicholson-type delay systems
with linear harvesting terms, \emph{Appl. Math. Modelling}, {\bf 36}(2012), 3289--3298.

\bibitem{Liu} B. Liu,  New results on the positive almost periodic solutions for a model of hematopoiesis,
\emph{Nonlinear Anal. RWA.}, {\bf 17}(2014), 252--264.

\bibitem{Long} F. Long, Positive almost periodic solution for a class of Nicholson's blowflies model
with a linear harvesting term, \emph{Nonlinear Anal. RWA.}, {\bf 13}(2012), 686--693.

\bibitem{GYZ} G. Liu, J. Yan and F. Zhang, Existence and global attractivity of unique positive periodic solution for a model of 
hematopoiesis, \emph{J. Math. Anal. Appl.}, {\bf334}(2007), 157--171.

\bibitem{LD} Q. L. Liu and H. S. Ding, Existence of positive almost-periodic solutions for a Nicholson's blowflies model,  
\emph{Electron. J. Diff. Eq.}, Vol. 2013 (2013), No. 56, 1--9.

\bibitem{MG} M. C. Mackey and L. Glass, Oscillation and chaos in physiological control system, 
\emph{Science}, {\bf 197} (1977), 287--289.

\bibitem{SP} A. M. Samoilenko and N. A. Perestyuk, 
\emph{Impulsive Differential Equations}, World Scientific, Singapore, 1995.

\bibitem{Sm} H. Smith, 
\emph{An Introduction to Delay Differential Equations with Applications to the Life Sciences},
Springer-Verlag, New York, 2011.

\bibitem{Stamov} G. Tr. Stamov, \emph{Almost Periodic Solutions of Impulsive
Differential Equations}, Springer-Verlag, Berlin, 2012.

\bibitem{WZ} X. Wang and H. Zhang, A new approach to the existence, nonexistence and uniqueness of 
positive almost periodic solution for a model of hematopoiesis, 
\emph{Nonlinear Anal. RWA}, {\bf 11}(2010), 60--66.

\bibitem{WLX} X. Wang, S. Li and D. Xu, Globally exponential stability of periodic solutions for 
impulsive neutral-type neural networks with delays, \emph{Nonlinear Dyn.}, {\bf64}(2011), 65--75.

\bibitem{WZh} X. Wei and W. Zhou, Uniqueness of positive solutions for an elliptic system arising
in a diffusive predator-prey model, 
\emph{Electron. J. Diff. Eq.}, Vol. 2013 (2013), No. 34, 1--4. 

\bibitem{WLZ} X. Wu, J. Li and H. Zhou, A necessary and sufficient condition for the
existence of positive periodic solutions of a model of hematopoiesis,
\emph{Comput. Math. Appl.}, {\bf 54}(2007), 840--849.

\bibitem{ZYW} H. Zhang, M. Yang and L. Wang, Existence and exponential convergence of the
positive almost periodic solution for a model of hematopoiesis, \emph{Appl. Math. Lett.}, {\bf26}(2013), 38--42. 

\end{thebibliography}
\end{document}